\documentclass{article}[12pt]
\usepackage{a4}
\usepackage[english]{babel}
\usepackage[T1]{fontenc}
\usepackage[all]{xy}
\usepackage{amssymb}
\usepackage{amsmath}
\usepackage{amsthm}
\usepackage{graphicx}
\newtheorem{thm}{Theorem}[section]
\newtheorem{prop}[thm]{Proposition}
\newtheorem{lem}[thm]{Lemma}

\newtheorem{defn}[thm]{Definition}

\newtheorem{rmk}[thm]{Remark}

\numberwithin{equation}{section}   

\newcommand{\rdbb}{\mathbb{R}^d}
\newcommand{\lra}{\longrightarrow}

\newcommand{\rbb}{\mathbb{R}}
\newcommand{\kd}{\mathcal{K}_d}

\newcommand{\as}{\mathcal{A}s}
\newcommand{\mcalo}{\mathcal{O}}

\newcommand{\unit}{\textbf{1}}

\newcommand{\mbbc}{\mathcal{C}}
\newcommand{\mcau}{\mathcal{U}}

\title{ \textbf{Symmetric multiplicative formality of the Kontsevich operad}}
\date{}

\author{Paul Arnaud Songhafouo Tsopméné}

\begin{document}
\maketitle

\begin{abstract}
In his famous paper entitled \lq\lq Operads and motives in deformation quantization\rq\rq{}, Maxim Kontsevich constructed (in order to prove the formality of the little $d$-disks operad) a topological operad, which is called in the literature the \textit{Kontsevich operad}, and which is denoted $\kd$ in this paper. This operad has a nice structure: it is a \textit{multiplicative symmetric operad}, that is, it comes with a morphism from the symmetric associative operad. There are many results in the literature regarding the formality of $\kd$. It is well  known (by Kontsevich) that $\kd$ is formal over reals as a symmetric operad. It is also well known (independently by  Syunji Moriya and the author) that $\kd$ is formal as a multiplicative nonsymmetric operad. In this paper, we prove that the Kontsevich operad is formal over reals as a multiplicative symmetric operad, when $d \geq 3$. 
\end{abstract}

\tableofcontents

\section{Introduction}

Let $d \geq 1$ be the dimension of the ambient space. The object we look at in this paper is the \textit{Kontsevich operad} that we denote $\kd := \{\kd(n)\}_{n \geq 0}$.  Recall that $\kd(0)$ and $\kd(1)$ are the one-point spaces, and for $n \geq 2$, $\kd(n)$ is the Kontsevich compactification of the configuration space of $n$ points in $\rdbb$.  More precisely, let $\text{Conf}(n, \rbb^d)$ denote the ordered configuration space of $n$ points in $\rdbb$ with the subspace topology. Define a map 
\[
\tau \colon \text{Conf}(n, \rdbb) \lra \prod_{1 \leq i < j \leq n} S^{d-1} \ \ \text{by} \ \  \tau(x_1, \cdots, x_n) = \left(\frac{x_j -x_i}{\left\|x_j - x_i\right\|} \right)_{1\leq i < j \leq n}.
\]
By definition $\kd(n)$ is the closure of the image of $\tau$. That is, $\kd(n) = \overline{\text{im}}(\tau)$. In particular, when $d=1$, one has 
\begin{eqnarray} \label{ko_defn}
\mathcal{K}_1(n) = *, n \geq 0,
\end{eqnarray}
where $*$ stands for the one-point topological space. Intuitively, one should think an element $x \in \kd(n)$ as a \lq\lq virtual\rq\rq{} configuration of $n$ points in which it is possible to get two or more points that are infinitesimally close to each other in such a way that the direction between any two of them is always well recorded. From this point of view, given two virtual configurations  $x=(x_1, \cdots, x_n) \in \kd(n)$ and $y=(y_1, \cdots, y_m) \in \kd(m)$, given an integer $1 \leq i \leq n$, we can form a new virtual configuration $x \circ_i y \in \kd(n+m-1)$ of $n+m-1$ points by \lq\lq substituting\rq\rq{} the point $x_i$ by the configuration $y$ made infinitesimal, and by keeping the other points  of the configuration $x$. This defines an operad structure on $\kd = \{\kd(n)\}_{n \geq 0}$. For a more precise definition of this structure, we refer the reader to the paper \cite[Section 4]{sin06} of Sinha. 

This latter author shows \cite{sin04} that the operad $\kd$ is weakly equivalent, as a topological operad, to the little $d$-disks operad. He also proves \cite{sin06} that $\kd$  is a  multiplicative symmetric operad.  Recall that a topological  operad $O$ is  said to be a  \textit{multiplicative symmetric operad} if it is a symmetric operad, and if there is a morphism $\Sigma[\as] \lra O$ of operads from the symmetric associative operad to $O$ that respects the symmetric structures. Here $\Sigma[\as](n)$ is the coproduct $\Sigma[\as](n) = \coprod_{\sigma \in \Sigma_n} *$, where $\Sigma_n$ stands for the symmetric group on $n$ letters.  Since this paper is devoted to the study of the formality of $\kd$, we need the following definition.


\begin{defn}
Let $O$ be a topological  operad. We say that $O$ is \emph{stable formal} or just \emph{formal} over a field $\mathbb{K}$ if there exists a zigzag
\begin{eqnarray} \label{zigzag_intro}
\xymatrix{ C_* (O; \mathbb{K}) & \ar[l]_-{\sim} \cdots \ar[r]^-{\sim} & H_* (O; \mathbb{K})}
\end{eqnarray}
of quasi-isomorphisms between the singular chain of $O$ and its homology.
\begin{enumerate}
\item[$\bullet$] If $O$ is equipped with a symmetric structure, and if all morphisms in (\ref{zigzag_intro}) are of symmetric operads, we say that $O$ is  \emph{formal as a symmetric operad}.
\item[$\bullet$] If $O$ is a multiplicative nonsymmetric operad (that is, there is a morphism $\{*\}_{n \geq 0}=\as \lra O$ from the associative operad to $O$), and if (\ref{zigzag_intro}) holds in the category of multiplicative nonsymmetric operads, we say that $O$ is \emph{formal as a multiplicative operad}.
\item[$\bullet$] If $O$ is a multiplicative symmetric operad, and if (\ref{zigzag_intro}) holds in the category of multiplicative symmetric operads, we say that $O$ is \emph{formal as a multiplicative symmetric operad}.
\end{enumerate}
\end{defn}

There are many formality theorems regarding $\kd$ in the literature, but we are only interested in the following two known results. The first one is due to Kontsevich \cite{kont99}, and it says that the operad $\kd$ is formal over reals as a symmetric operad. The second one, independently due to Syunji Moriya \cite{moriya12} and the  author \cite{songhaf12}, states that $\kd$ is formal over reals as a multiplicative operad. One natural question one can ask is to know whether those two results immediately imply that $\kd$ is formal as a multiplicative symmetric operad. As we will see in this text, the answer to this question is not immediate, and requires a slightly more careful analysis.   Our main result is the following.

\begin{thm} \label{for_kont_thm}
For $d \geq 3$, the Kontsevich operad $\kd$ is formal over reals as a multiplicative symmetric operad.  
\end{thm}

As we said before, the multiplicative formality of $\kd$ was independently discovered by Moriya \cite{moriya12} and by the author \cite{songhaf12}. These two authors used this formality to  explicitly compute the natural Gerstenhaber algebra structure on the homology of the space of long knots. 



\textbf{Outline of the paper} In Section~\ref{equiv_mult_sym_operads} we prove a crucial Lemma~\ref{sym_lem2}, which will be the main ingredient in proving Theorem~\ref{for_kont_thm}.

In Section~\ref{for_kont_section} we prove Theorem~\ref{for_kont_thm} by applying Lemma~\ref{sym_lem2} to the specific zigzag between $C_*(\kd; \rbb)$ and its homology $H_*(\kd; \rbb)$. 


\section{Equivalences of multiplicative symmetric operads} \label{equiv_mult_sym_operads}

The goal here is to prove Lemma~\ref{sym_lem2}, which is the main result of this section. This lemma will be the key ingredient in the proof of Theorem~\ref{for_kont_thm}, which will be done in  Section~\ref{for_kont_section}. For our purposes, we need to start with the following definition. 

\begin{defn}\label{monoidalmodelcat_def}
Let $\mathcal{C} := (\mbbc, \otimes, \unit)$ be a symmetric monoidal category. 
\begin{enumerate}
\item[(i)] The category $\mathcal{C}$ is called \emph{symmetric monoidal model category} \index{symmetric monoidal model category} if it is equipped with a model structure such that the following axioms hold in $\mathcal{C}$.\\
\textit{Unit axiom}: the unit object, \emph{\textbf{1}}, is cofibrant in $\mathcal{C}$.\\
\textit{Pushout-product axiom}: the natural morphism
$(i_*, j_*): A \otimes D \oplus_{A \otimes C} B \otimes C \longrightarrow B \otimes D$
induced by cofibrations $i: \xymatrix{A\;\ar@{>->}[r] & B}$ and $j: \xymatrix{C\;\ar@{>->}[r] & D}$ forms a cofibration, respectively an acyclic cofibration if $i$ or $j$ is also acyclic.
\item[(ii) ]A model category  is said to be \emph{cofibrantly generated} \index{cofibrantly generated model category} if it is equipped with a set of generating cofibrations $\mathcal{I}$, and a set of generating acyclic cofibrations $\mathcal{J}$, such that
\begin{enumerate}
\item[$\bullet$] the fibrations are characterized by the right lifting property with respect to the acyclic generating cofibrations $j \in \mathcal{J}$, and 
\item[$\bullet$] the acyclic fibrations are characterized by the right lifting property with respect to the generating cofibrations $i \in \mathcal{I}$.
\end{enumerate}
\end{enumerate} 
\end{defn}

Throughout the section we reserve the letter $\mbbc:= (\mbbc, \otimes, \unit)$ for the base category over which our operads will be defined. It is a symmetric monoidal model category that is cofibrantly generated. 

\begin{rmk} \label{small_limitcolimit_rmk}
Since $\mathcal{C}$ is in particular a model category, it follows that it has all small limits and colimits (see \cite[Definition 1.1.4]{hovey99}).  
\end{rmk}

In order to prove Theorem~\ref{for_kont_thm}, we need  Lemma~\ref{sym_lem1} below, which involves the classical adjunction between the categories $\mcalo P_{ns} (\mbbc)$ and $\mcalo P_s (\mbbc)$ of nonsymmetric operads and symmetric operads respectively. 
Let us first recall the definition of a symmetric operad and of a nonsymmetric operad in $\mbbc$. A \textit{symmetric sequence} in $\mbbc$ consists of a sequence $\{X_n\}_{n \geq 0}$ in which  each $X_n$ is an object in $\mbbc$ equipped with an action of  the symmetric group $\Sigma_n$. By action we mean a morphism $\alpha_{X_n} \colon \Sigma_n \otimes X_n \lra X_n$ such that the obvious diagrams 

\[
\begin{minipage}{0.5\textwidth}
\xymatrix{\{e_{\Sigma_n}\} \otimes X_n \ar[rr]^-{i \otimes \mbox{id}_{X_n}} \ar[rrd]_{\cong} & &\Sigma_n \otimes X_n \ar[d]^{\alpha_{X_n}} \\                                                                               
                                                                                            & & X_n }
\end{minipage}%
\begin{minipage}[t]{1.5cm}
\includegraphics[scale=0.1]{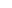}
\end{minipage}
\begin{minipage}{0.5\textwidth}
\xymatrix{\Sigma_n \otimes (\Sigma_n \otimes X_n) \ar[rr]^-{\circ \otimes \mbox{id}_{X_n}} \ar[d]_-{\mbox{id}_{\Sigma_n} \otimes \alpha_{X_n}} & & \Sigma_n \otimes X_n \ar[d]^-{\alpha_{X_n}} \\
\Sigma_n \otimes X_n     \ar[rr]_-{\alpha_{X_n}} & &  X_n }
\end{minipage}
\]

commute. Here,
\begin{enumerate}
\item[$\bullet$] $\Sigma_n \otimes X_n$ is the coproduct  $\Sigma_n \otimes X_n= \coprod_{\sigma \in \Sigma_n} X_n$, \footnote{This coproduct holds in $\mathcal{C}$, and exists because of Remark~\ref{small_limitcolimit_rmk}}
\item[$\bullet$] $i \colon \{e_{\Sigma_n} \} \hookrightarrow \Sigma_n$ is the inclusion of the unit, 
\item[$\bullet$] $\Sigma_n \otimes (\Sigma_n \otimes X_n)$ is naturally isomorphic to $(\Sigma_n \times \Sigma_n) \otimes X_n$, and
\item[$\bullet$]  \lq\lq$\circ$\rq\rq{} is the composition in $\Sigma_n$.  
\end{enumerate}
A \textit{symmetric operad} in $\mbbc$ consists of a symmetric sequence $O = \{O(n)\}_{n \geq 0}$ endowed with a unit $\unit  \lra O(1)$, and a collection of morphisms  
\[O(k) \otimes O(i_1) \otimes \cdots \otimes O(i_k) \lra O (i_1+\cdots + i_k) \]
that satisfy natural equivariance properties, unit and associative axioms (May's axioms, see \cite[Definition 1.1]{may72}). If we forget the symmetric structure, we obtain the so called \textit{nonsymmetric operad}. Symmetric operads (respectively nonsymmetric operads) form a category that we denote by $\mcalo P_s (\mbbc)$ or just by $\mcalo P_s$  (respectively by $\mcalo P_{ns} (\mbbc)$ or just by $\mcalo P_{ns}$).  These two categories are related by an adjunction 
\begin{eqnarray} \label{adjunction_operads}
\Sigma [-] \colon \mcalo P_{ns} (\mbbc) \rightleftarrows \mcalo P_s (\mbbc) : \mathcal{U},
\end{eqnarray}
where $\mathcal{U}$ is the forgetful functor, and $\Sigma[-]$ is defined as $\Sigma[P](n) = \Sigma_n \otimes P(n)$. Therefore, there is a bijection 
$$\Psi \colon \underset{\mcalo P_{ns}}{\mbox{Hom}}(P, \mathcal{U}(Q)) \stackrel{\cong}{\lra} \underset{\mcalo P_s}{\mbox{Hom}} (\Sigma[P], Q).$$ 
Recall that this bijection sends a morphism $f \colon P \lra \mathcal{U}(Q)$ to $\Psi_f \colon \Sigma[P] \lra Q$ defined as the composition 
\[\xymatrix{\Sigma[P] \ar[rr]^-{\Sigma[f]} & &  \Sigma[\mathcal{U}(Q)] \ar[r]^-{\eta_Q} & Q},\]
where $\eta = \{\eta_Q \colon \Sigma[\mcau(Q)] \lra Q\}_{Q \in \mcalo p_s}$ is the counit of the adjunction (\ref{adjunction_operads}). That is, $\Psi_f = \eta_Q \circ \Sigma[f]$. Here the counit is defined as follows. For all $n \geq 0$, the morphism 
\[(\eta_Q)_n = \alpha_{Q(n)} \colon \Sigma_n \otimes \mcau (Q) (n) \lra Q(n)
\]
is provided by the action of $\Sigma_n$ on $Q(n)$ since   $\mcau(Q)(n) = Q(n)$ by the definition of the forgetful functor, and since $Q$ is in particular a symmetric sequence. 

\begin{lem} \label{sym_lem1}
Let $g \colon Q \lra Q'$ be a morphism in $\mcalo P_s(\mbbc)$. Let
\begin{eqnarray} \label{square_lem1}
\xymatrix{\mathcal{U}(Q) \ar[rr]^-{\mcau(g)} & & \mcau (Q') \\
            P \ar[u]^{f} \ar[rr]^-h & & P' \ar[u]_{f'}  }
\end{eqnarray}						
be a commutative square in the category $\mcalo P_{ns} (\mbbc)$. Then the diagram
\[\xymatrix{Q \ar[rr]^-g & & Q' \\
            \Sigma[P] \ar[u]^{\Psi_f} \ar[rr]^-{\Sigma[h]} & & \Sigma[P'] \ar[u]_{\Psi_{f'}}   } \]
is a commutative square in $\mcalo P_{s} (\mbbc)$. 											
\end{lem}

\begin{proof}
The result comes from the fact that the following two squares 
\[
\begin{minipage}{0.5\textwidth}
\xymatrix{\Sigma[\mcau(Q)] \ar[rr]^-{\Sigma[\mcau(g)]} & & \Sigma[\mcau(Q')] \\
            \Sigma[P] \ar[u]^{\Sigma[f]} \ar[rr]^-{\Sigma[h]} & & \Sigma[P'] \ar[u]_{\Sigma[f']} }
\end{minipage}%
\begin{minipage}[t]{1.5cm}
\includegraphics[scale=0.1]{blanc}
\end{minipage}
\begin{minipage}{0.5\textwidth}
\xymatrix{Q \ar[rr]^-g & & Q' \\
\Sigma[\mcau(Q)] \ar[rr]^-{\Sigma[\mcau(g)]} \ar[u]^{\eta_Q} & & \Sigma[\mcau(Q')] \ar[u]_{\eta_{Q'}}}
\end{minipage}
\]
commute. The first square commutes because it is the image of the commutative square (\ref{square_lem1})  under the functor $\Psi$, and the second one commutes because $g_n  \colon Q(n) \lra Q'(n)$ is $\Sigma_n$-equivariant (since , by hypothesis, the morphism $g$ is of symmetric operads) for all $n$.  

\end{proof}

Now,  to state and prove the most important result (Lemma~\ref{sym_lem2} below) of this section, we recall some notations and some classical facts. The first notation is that of a symmetric operad $\as = \{\as(n)\}_{n \geq 0}$ in $\mbbc$, which is a nonsymmetric operad. It is defined by $\as (n) = \unit$ for all $n$. Recall that $\unit$ is the unit for the tensor product in $\mbbc$.  

\begin{defn}
A \emph{multiplicative symmetric operad} is a symmetric operad $Q$ together with a morphism $\Sigma[\as] \lra Q$ in the category of symmetric operads. Similarly, a \emph{multiplicative operad} is a nonsymmetric operad $\mathcal{O}$ toghether with a morphism $\as \lra \mathcal{O}$ in $\mcalo P_{ns}$. 
\end{defn}

The second notation is $\mbbc^{\mathbb{N}}$, which represents the category of sequences $X = \{X(n)\}$ of objects in $\mbbc$. This category is endowed with the cofibrantly generated model structure in which weak equivalences, fibrations and cofibrations are all levelwise.  Of course a generating cofibration $f \colon X \lra Y$ of $\mbbc^{\mathbb{N}}$ is a collection $f =\{f_n \colon X_n \lra Y_n \}_{n \geq 0}$ where each $f_n$ is a generating cofibration of $\mathcal{C}$. The same observation holds for generating acyclic cofibrations of $\mbbc^{\mathbb{N}}$.  As a consequence of the model structure of $\mbbc^{\mathbb{N}}$, a sequence  $X = \{X_n\}$ is cofibrant if and only if  every $X_n$ is cofibrant in $\mathcal{C}$. The third and the last notation we will need is $\Sigma \mbbc^{\mathbb{N}}$, which denotes the category of symmetric sequences in $\mbbc$. This latter category is also endowed with a model structure (also cofibrantly generated) induced by the adjunction
\[\mbox{Sym} \colon \mbbc^{\mathbb{N}} \rightleftarrows \Sigma \mbbc^{\mathbb{N}} : U,\]
where $U$ is of course the forgetful functor, and $\mbox{Sym}$ is the functor defined  as $\mbox{Sym}(P)(n) = \Sigma_n \otimes P(n)$. Recall that a morphism $f \colon P \lra Q$ in $\Sigma \mbbc^{\mathbb{N}}$ is a weak equivalence (respectively a fibration) if $U(f)$ is a weak equivalence in $\mbbc^{\mathbb{N}}$ (respectively $U(f)$ is a fibration in $\mbbc^{\mathbb{N}}$). 

\begin{rmk} \label{cofibrant_rmq} The following are easy to establish. 
\begin{enumerate}
\item[(i)] If $X$ is a cofibrant object in $\mbbc^{\mathbb{N}}$, then so is $\mbox{Sym}(X)$. More generally, the functor $\mbox{Sym}$ preserves cofibrations (this easily follows from adjunction relations).  
\item[(ii)] The associative operad $\as$ is cofibrant as an object in $\mbbc^{\mathbb{N}}$. Indeed, the unit $\unit$ in $\mbbc$ is cofibrant by the unit axiom from Definition~\ref{monoidalmodelcat_def}~(i). 
\item[(iii)] From (i) and (ii), it follows that the symmetric operad $\Sigma[\as]$ is cofibrant as an object in $\Sigma \mbbc^{\mathbb{N}}$. 
\end{enumerate}
\end{rmk}

Before stating Lemma~\ref{sym_lem2}, we mention that the model structure  we consider on $\mcalo P_s (\mbbc)$  (respectively on $\mcalo P_{ns} (\mbbc)$) is the one from \cite[Theorem 12.2.A]{fress08} (respectively the one from the  nonsymmetric version of \cite[Theorem 12.2.A]{fress08}). Notice that the category $\mcalo P_s (\mbbc)$ usually has only semi-model structure by \cite[Theorem 12.2.A]{fress08}. That is, $\mcalo P_s (\mbbc)$ satisfies all axioms for a model category, including the lifting and the factorization axiom, but only for morphisms $f \colon P \lra Q$ whose the domain $P$ is cofibrant as an object of $\Sigma \mathcal{C}^{\mathbb{N}}$. Particularly if $P \in \mcalo P_s (\mbbc)$ is cofibrant as a symmetric sequence,  then any morphism $f \colon P \lra Q$ of symmetric operads has a factorization $f = \eta_2\eta_1$ in $\mcalo P_s (\mbbc)$, where $\eta_1$ is a cofibration and $\eta_2$ is an acyclic fibration. The same observation holds for the category $\mcalo P_{ns} (\mbbc)$.

\begin{lem} \label{sym_lem2}
Let $g \colon R \lra Q$ and $g' \colon R \lra Q'$ be two morphisms of symmetric operads, and let 
\begin{eqnarray} \label{diagram_lem2}
 \xymatrix{Q  & R \ar[l]_-{\sim}^g \ar[r]^{\sim}_{g'}  & Q' \\
                   \Sigma[\as]  \ar[u]^{f}  & \Sigma[\mathcal{A}] \ar[l]_{\sim}^-{h} \ar[r]^{\sim}_-{h}  \ar[u]^{\eta} &  \Sigma[\as]. \ar[u]_{f'}    }  
 \end{eqnarray}
 be a commutative diagram in $\mcalo P_{s}(\mbbc)$. Assume that  $\mathcal{A}$ is cofibrant as an object of  $\mbbc^{\mathbb{N}}$.  Then there is a commutative diagram 
   \[ \xymatrix{Q  & \widetilde{R} \ar[l]_-{\sim} \ar[r]^{\sim}  & Q' \\
                   \Sigma[\as]  \ar[u]^{f}  & \Sigma[\as] \ar[l]_{\sim}^{=} \ar[r]^{\sim}_{=}  \ar[u] &  \Sigma[\as] \ar[u]_{f'}    }  
                    \]
 in the category    $\mcalo P_{s}(\mbbc)$ of symmetric operads.            
\end{lem}

\begin{proof}
The proof uses exactly the same techniques as that of the proof of \cite[Lemma 2.7]{songhaf12}, which regards nonsymmetric operads. To be more precise, one has the following commutative diagram (in the category $\mcalo P_{s}(\mbbc)$) constructed as follows. 

\begin{eqnarray} \label{diagram1_proof}
\xymatrix{Q   &  R \ar[l]_{\sim}^g  \ar[rrr]^-{\sim}_-{g'} & & & Q'  \\
              &                                            &  \overline{R} \ar[r]^{\sim}_{h_1} \ar@{->>}[lu]_{\sim}^{\eta_2} & \widetilde{R} \ar[ru]_{h_3}^{\sim} &                  \\
          \Sigma[\as] \ar[uu]^{f}  & \Sigma[\mathcal{A}] \; \; \; \ar@{>->}[ru]^{\eta_1} \ar[l]^-{h}_{\sim} \ar[rrr]^-{\sim}_-{h}  \ar[uu]^{\eta} & & & \Sigma[\as] \ar[uu]_{f'} \ar[lu]_{h_2}   }
\end{eqnarray}

\begin{enumerate}
\item[$\bullet$] The morphisms $\eta_1$ and $\eta_2$ comes from the factorization axiom applied to $\eta$. One can apply this axiom here because  $\Sigma[\mathcal{A}]$ is cofibrant as an object in $\Sigma \mbbc^{\mathbb{N}}$. Indeed, by hypothesis, the operad $\mathcal{A}$  is cofibrant as an object in $\mbbc^{\mathbb{N}}$, which implies (by Remark~\ref{cofibrant_rmq}~(i)) that  $\Sigma[\mathcal{A}]$ is cofibrant in $\Sigma \mbbc^{\mathbb{N}}$.
\item[$\bullet$] The object $\widetilde{R}$ is the pushout of the diagram formed by morphisms $\eta_1$ and $h$. 
\item[$\bullet$] The morphism $h_1$ is a weak equivalence because of the following. Recall first that the axiom of relative properness for symmetric operads \cite[Theorem 12.2.B]{fress08} says that if $P_1$ and $P_2$ are two objects in $\mcalo P_s (\mbbc)$ that are cofibrant as objects in $\Sigma \mbbc^{\mathbb{N}}$, then the pushout of a weak equivalence along a cofibration 
\[
\xymatrix{P_1 \; \ar@{>->}[r]   \ar[d]^{\sim} &    P_3 \ar@{.>}[d]  \\
          P_2  \ar@{.>}[r]        &  P_4 }
\]
gives a weak equivalence $\xymatrix{P_3 \ar[r]^{\sim}  & P_4}$. Since the operad $\Sigma[\mathcal{A}]$ is cofibrant as an object in $\Sigma \mbbc^{\mathbb{N}}$ (by hypothesis and by Remark~\ref{cofibrant_rmq}~(i)), and since $\Sigma[\as]$ is also cofibrant as an object in $\Sigma \mbbc^{\mathbb{N}}$ (by Remark~\ref{cofibrant_rmq}~(iii)), it follows by the axiom of relative properness that $h_1$ is a weak equivalence. 
\item[$\bullet$] The morphism $h_3 \colon \widetilde{R} \lra Q'$ comes from the universal property of the pushout. Moreover, it is a weak equivalence because of the two-out-of-three axiom \footnote{Two-out-of-three axiom: Let $f$ and $g$ be composable morphisms. If any two among $f$, $g$, and $gf$ are weak equivalences, then so is the third} since $g'\eta_2 = h_3h_1$, and $g'\eta_2$ and $h_1$ are weak equivalences. 
\end{enumerate}
One has also a commutative diagram
\begin{eqnarray} \label{diagram2_proof}
\xymatrix{
    \Sigma[\mathcal{A}] \; \ar@{>->}[r]^-{\eta_1} \ar[d]_{h}^{\sim} & \overline{R}  \ar@/^/[rdd]^{g \eta_2}_{\sim} \ar[d]_{\sim}^{h_1} \\
     \Sigma[\as] \ar@/_/[rrd]_{f} \ar[r]^-{h_2} & \widetilde{R} \ar@{.>}[rd]^-{\sim}_-{h_4}\\
     &&  Q}
\end{eqnarray}
in which the morphism $h_4 \colon \widetilde{R} \lra Q$ comes from the universal property of pushout. This latter morphism is a weak equivalence because of the two out of three axiom since $g\eta_2$ and $h_1$ are weak equivalences. 

Now  from diagrams (\ref{diagram1_proof}) and (\ref{diagram2_proof}), we obtain 

\[
\xymatrix{Q    &   \widetilde{R} \ar[l]_-{\sim}^-{h_4} \ar[r]^{\sim}_{h_3}  &  Q'  \\
          \Sigma[\as]  \ar[u]^{f}    &  \Sigma[\as] \ar[l]_-{\sim}^-{=} \ar[u]^{h_2}  \ar[r]^-{\sim}_-{=}  &  \Sigma[\as], \ar[u]_{f'}}
\]

which is a commutative diagram in $\mcalo P_{s}(\mbbc)$. This ends the proof.

\end{proof}

\section{Formality of the Kontsevich operad as a multiplicative symmetric operad} \label{for_kont_section}

The goal of this  section is to prove Theorem~\ref{for_kont_thm}  announced in the introduction.  The base category $\mbbc$ for operads is taken to be $\mbox{Ch}_{\rbb}$, the category of nonnegatively graded chain complexes over the ground field $\rbb$. This latter category is equipped with its standard symmetric monoidal model structure (see \cite[Theorem 2.3.11]{hovey99} for the model structure). That is, weak equivalences (quasi-isomorphisms) and fibrations (epimorphisms) are all level wise. A chain map $f_* \colon A_* \lra B_*$ is a cofibration if for any $n \geq 0$, the map $f_n \colon A_n \lra B_n$ is an injection with projective cokernel. It follows that a chain complex $A_*$ is cofibrant if and only if each $A_n$ is a projective module.  

Before starting the proof of Theorem~\ref{for_kont_thm}, we recall the specific zigzag 

\begin{eqnarray}\label{zigzag_formality}
\xymatrix{C_* (\mathcal{K}_d) & C_*(\mathcal{F}_d) \ar[l]_{\sim} \ar[r]^{\sim} & \mathcal{D}^{\vee}_d & H_*(\mathcal{K}_d) \ar[l]_{\sim}\\
              C_*(\mathcal{K}_1) \ar[u] & C_*(\mathcal{F}_1) \ar[l]_{\sim} \ar[r]^{\sim} \ar[u] & \mathcal{D}_1^{\vee} \ar[u] & H_*(\mathcal{K}_1). \ar[l]_{\sim} \ar[u]}
\end{eqnarray}	 

between the singular chain $C_* (\kd)$ and the homology $H_*(\kd)$ of the Kontsevich operad. This zigzag holds  in the category $\mcalo P_{ns} (\mbox{Ch}_{\rbb})$ of nonsymmetric operads in $\mbox{Ch}_{\rbb}$, and  was explicitly built by Lambrechts and Voli\'c in \cite{lam_vol} (see \cite[Section 9]{lam_vol} and \cite[Section 8]{lam_vol} for the construction of $C_*(\mathcal{F}_d) \stackrel{\sim}{\longrightarrow} \mathcal{D}^{\vee}_d$ and $H_*(\mathcal{K}_d)\stackrel{\sim}{\longrightarrow} \mathcal{D}^{\vee}_d$ respectively).   They built (\ref{zigzag_formality}) in order to develop the Kontsevich's proof \cite{kont99} for the formality of the little $d$-disks operad. Recall that $\mathcal{F}_d$ is the Fulton-MacPherson operad, which is well explained in \cite[Chapter 5]{lam_vol}. Intuitively, an element  $x \in \mathcal{F}_d(n)$ can be viewed as a \lq\lq virtual\rq\rq{} configuration of $n$ points (in $\rdbb$ of course) in which points are allowed to be infinitesimally close to each other while the directions and the relative distances of their approach are recorded. For instance, when $d =1$, one has $\mathcal{F}_1(0) = \mathcal{F}_1(1) = \mathcal{F}_1(2) = *$, while $\mathcal{F}_1(3) \cong [0, 1]$. In general $\mathcal{F}_1(n)$ is known to be the Stasheff associahedron \cite[Subsection 4.4]{sin04}, which implies the following. 
\begin{prop} \label{fo_prop}
For all $n \geq 0$, the topological space $\mathcal{F}_1(n)$ is contractible. 
\end{prop} 

Recall also the operad $\mathcal{D}^{\vee}_d$, which is  the dual of the cooperad $\mathcal{D}_d$ from \cite[Chapter 6 and Chapter 7]{lam_vol}. For $d \geq 2$, $\mathcal{D}^{\vee}_d(n)$ is a chain complex generated by certain graphs with $n$ labeled external vertices. For $d =1$, $\mathcal{D}^{\vee}_1(n)$ is defined as 
\begin{eqnarray} \label{do_defn}
\mathcal{D}^{\vee}_1(n) = H_* (\mathcal{F}_1(n)), n \geq 0.
\end{eqnarray}
Recalling the definition of the operad $\mathcal{K}_d$ from the introduction, we have the following remark. 
\begin{rmk} \label{symmetric_rmk}
Clearly, the operads $\mathcal{K}_d$ and $\mathcal{F}_d$ are both symmetric. The symmetric structure is the canonical one, which is induced by the permutation of points of a configuration. The operad $\mathcal{D}^{\vee}_d$ has the same kind of symmetric structure (by permuting external vertices instead).  
\end{rmk}

The idea of the following proof is the same as that of \cite[Theorem 1.3]{songhaf12}, except that here we work with symmetric operads. 

\begin{proof}[\textbf{Proof of Theorem~\ref{for_kont_thm}}]
We start with the following two observations. The first one says that all the operads appearing  in the first row of (\ref{zigzag_formality}) are symmetric. This immediately comes from Remark~\ref{symmetric_rmk} and the fact that the functors $C_*(-)$ and $H_*(-)$ are symmetric monoidal.  The second observation consists of the following three things:
\begin{enumerate}
\item[$\bullet$] $C_*(\mathcal{K}_1)$ and $H_*(\mathcal{K}_1)$ are both the associative operad $\as$ in $\mbox{Ch}_{\rbb}$ since $C_*(-)$ and $H_*(-)$ are symmetric monoidal functors and $\mathcal{K}_1$ is the topological nonsymmetric associative operad by (\ref{ko_defn}), 
\item[$\bullet$] $\mathcal{D}_1^{\vee}$ is the associative operad by (\ref{do_defn}) and Proposition~\ref{fo_prop}, and 
\item[$\bullet$] $C_*(\mathcal{F}_1)$ is not the associative operad since $\mathcal{F}_1(3)$ is homeomorphic to the unit interval, and $C_*([0,1])$ is not the unit for the usual tensor product in $\mbox{Ch}_{\rbb}$. 
\end{enumerate}

From these two observations, if we denote  $C_*(\mathcal{F}_1)$ by $\mathcal{A}$, one can rewrite (\ref{zigzag_formality}) as 
\begin{eqnarray}\label{zigzag_formality_sym}
\xymatrix{\mathcal{U}(C_* (\mathcal{K}_d)) & \mathcal{U} (C_*(\mathcal{F}_d)) \ar[l]_{\sim} \ar[r]^{\sim} & \mathcal{U} (\mathcal{D}^{\vee}_d) & \mathcal{U} (H_*(\mathcal{K}_d)) \ar[l]_{\sim}\\
              \as \ar[u] & \mathcal{A} \ar[l]_{\sim} \ar[r]^{\sim} \ar[u] & \as \ar[u]& \as, \ar[l]_{\sim} \ar[u]}
\end{eqnarray}	 
where $\mathcal{U}(-)$ is the forgetful functor from (\ref{adjunction_operads}). Applying now Lemma~\ref{sym_lem1} to (\ref{zigzag_formality_sym}) we get
\begin{eqnarray}\label{zigzag2_formality_sym}
\xymatrix{C_* (\mathcal{K}_d) & C_*(\mathcal{F}_d) \ar[l]_{\sim} \ar[r]^{\sim} & \mathcal{D}^{\vee}_d & H_*(\mathcal{K}_d) \ar[l]_{\sim}\\
              \Sigma[\as] \ar[u] & \Sigma[\mathcal{A}] \ar[l]_{\sim} \ar[r]^{\sim} \ar[u] & \Sigma[\as] \ar[u]& \Sigma[\as], \ar[l]_{\sim} \ar[u]}
\end{eqnarray}	 
which is a commutative diagram in $\mathcal{O}P_s(\mbox{Ch}_{\rbb})$. Since  $\mathcal{A} = C_*(\mathcal{F}_1)$ is cofibrant  as an object of $\mbox{Ch}_{\rbb}^{\mathbb{N}}$ \footnote{This is because each $C_*(\mathcal{F}_1(n))$ is cofibrant in $\mbox{Ch}_{\rbb}$ since every $C_k(\mathcal{F}_1(n)), k \geq 0,$  is a module over $\rbb$ (that is, a vector space), and  cofibrant objects in $\mbox{Ch}_{\rbb}$ are chain complexes of projective modules as we mentioned at the beginning of this section}, and since the two morphisms from $\Sigma[\mathcal{A}]$ to $\Sigma[\as]$ are the same, the desired result follows by applying  Lemma~\ref{sym_lem2} to the subdiagram 
\[
\xymatrix{C_* (\mathcal{K}_d) & C_*(\mathcal{F}_d) \ar[l]_{\sim} \ar[r]^{\sim} & \mathcal{D}^{\vee}_d \\
              \Sigma[\as] \ar[u] & \Sigma[\mathcal{A}] \ar[l]_{\sim} \ar[r]^{\sim} \ar[u] & \Sigma[\as] \ar[u] }
\] 
from (\ref{zigzag2_formality_sym}). This ends the proof. 
\end{proof}



\textsf{University of Regina, 3737 Wascana Pkwy, Regina, SK S4S 0A2, Canada\\
Department of Mathematics and Statistics\\}
\textit{E-mail address: pso748@uregina.ca}

\end{document}